\documentclass{proc-l}
\usepackage[mathscr]{eucal}
\usepackage[breaklinks=true]{hyperref}
\usepackage[dvipsnames]{xcolor}
\usepackage{enumerate}

\usepackage{amsmath,amsthm,amsfonts,amssymb,tocvsec2}
\usepackage{mathdots}

\theoremstyle{plain}
\newtheorem{theorem}{Theorem}[section]

\newtheorem{lemma}[theorem]{Lemma}

\newtheorem{utheorem}{\textrm{\textbf{Theorem}}}

\theoremstyle{definition}
\newtheorem{defn}[theorem]{Definition}

\newtheorem{rem}[theorem]{Remark}

\numberwithin{equation}{section}

\DeclareMathOperator{\adj}{adj}
\DeclareMathOperator{\rk}{rank}

\begin{document}
	\title[Sign regular matrices and variation diminution]{Sign regular matrices and variation diminution: single-vector tests and characterizations, following Schoenberg, Gantmacher--Krein, and Motzkin}
	
	\author{Projesh Nath Choudhury and Shivangi Yadav}
	\address[P.N.~Choudhury]{Department of Mathematics, Indian Institute of Technology Gandhinagar, Gujarat 382355, India}
	\email{\tt projeshnc@iitgn.ac.in}
	\address[S.~Yadav]{Department of Mathematics, Indian Institute of Technology Gandhinagar, Gujarat 382355, India; Ph.D. Student}
	\email{\tt shivangi.yadav@iitgn.ac.in, shivangi97.y@gmail.com}
	
	\date{\today}
	
\begin{abstract}
	Variation diminution (VD) is a fundamental property in total positivity theory, first studied in 1912 by Fekete--P\'olya for one-sided P\'olya frequency sequences, followed by Schoenberg, and by Motzkin who characterized sign regular (SR) matrices using VD and some rank hypotheses. A classical theorem by Gantmacher--Krein characterized the  strictly sign regular (SSR) $m \times n$ matrices for $m>n$ using this property.
	
	In this article we strengthen these results by characterizing all $m \times n$ SSR matrices using VD. We further characterize strict sign regularity of a given sign pattern in terms of VD together with a natural condition motivated by total positivity. We then refine Motzkin's characterization of SR matrices by omitting the rank condition and specifying the sign pattern. This concludes a line of investigation on VD started by Fekete--P\'olya [\textit{Rend.\ Circ.\ Mat.\ Palermo} 1912] and continued by Schoenberg [\textit{Math.\ Z.}\ 1930], Motzkin [PhD thesis, 1936], Gantmacher--Krein [1950 book], Brown--Johnstone--MacGibbon [\textit{J.\ Amer.\ Stat.\ Assoc.}\ 1981], and Choudhury [\textit{Bull.\ London Math.\ Soc.}\ 2022, \textit{Bull.\ Sci.\ Math.}\ 2023].
	
	In fact we show stronger characterizations, by employing single test vectors with alternating sign coordinates -- i.e., lying in the alternating bi-orthant. We also show that test vectors chosen from any other orthant will not work.
\end{abstract}
	
	
	
	\maketitle
	
	\vspace*{-11mm}
	\settocdepth{section}
	
	\section{Introduction and main results} \label{section introduction}
	Given integers $m,n\geq k\geq1$, an $m\times n$ real matrix $A$ is \textit{strictly sign regular of order $k$} (SSR$_k$) if there exists a sequence of signs $\epsilon_r\in\{1,-1\}$ such that every $r\times r$ minor of $A$ has sign $\epsilon_r$ for $1\leq r\leq k$.  A matrix $A$ is \textit{strictly sign regular} (SSR) if $A$ is SSR$_k$ for $k=\mathrm{min}\{m,n\}$. If minors are allowed to also vanish, then $A$ is correspondingly said to be \textit{sign regular of order $k$} (SR$_k$) and \textit{sign regular} (SR). For an SSR (respectively SR) matrix $A$, if $\epsilon_r=1$ for all $r\geq1$, then $A$ is said to be \textit{totally positive} (TP) (respectively \textit{totally non-negative} (TN)). These matrices have applications in numerous subfields: analysis, approximation theory, cluster algebras, combinatorics, differential equations, Gabor analysis, integrable systems, matrix theory, probability and statistics, Lie theory, and representation theory \cite{Ando87,BGKP20,BFZ96,Bre95,fallat-john,FZ00,FZ02,gantmacher-krein,GW96,GRS18,Karlin64,Karlinsplines,Khare20,KW14,Lo55,Lu94,pinkus,Ri03,Schoenberg46,S55,Whitney}.
	
	SR and SSR matrices $A$ enjoy the variation diminishing (VD) property, which states that the number of sign changes of $A\mathbf{x}$ is bounded above by the number of sign changes of $\mathbf{x}$, where  $\mathbf{x}$ is any vector. Variation diminution is considered to have originated from the famous 1883 memoir of Laguerre \cite{Laguerre}, but the phrase ``variation diminishing" (``variationsvermindernd" in German) was coined by P\'{o}lya when Fekete showed (using P\'{o}lya frequency sequences in \cite{FP12}) the following result of Laguerre on the variations/sign changes in the coefficients of power series: \textit{if $f(x)$ is a polynomial and $s\geq0$, then the variations var($e^{sx}f(x)$) in the Maclaurin coefficients of $e^{sx}f(x)$ do not increase in $0\leq s<\infty$ and hence are bounded above by var($f$)$<\infty$}. 
	For more examples and instances of the variation diminishing property in classical analysis, we refer to \cite{PS1925}; there are also applications to the theory of splines -- early papers include \cite{Curry, SW53}. 
	
	We now look back on some fundamental contributions in this direction: In 1930, Schoenberg \cite{S30} initiated the study of the variation diminishing property of a matrix. He showed that SR matrices satisfy this property. In his 1936 Ph.D. thesis \cite{Mot36}, Motzkin then characterized all $m \times n$ matrices satisfying the VD property as well as  certain rank-constraints for technical reasons (e.g. $\rk(A)=n$), by showing that they are precisely SR matrices. Thereafter, in 1950, Gantmacher--Krein  \cite{GK50} characterized $m\times n$ SSR matrices for $m >n$ 
	using the variation diminishing property. One of the main theorems in this article extends their result 
	by removing the size restriction $m>n$.
	
	The following definitions and notations are essential in order to formulate our results.
	\begin{defn}
		Throughout this paper, $m,n\geq1$ will denote fixed but arbitrary integers.
		\begin{enumerate}
			\item For $\mathbf{x}=(x_1,\ldots,x_n)^T\in\mathbb{R}^n$ we denote by $S^-(\mathbf{x})$ the number of sign changes in the sequence $x_1,\ldots,x_n$ after discarding all zero entries. (In this paper we will interchangeably use the words ``components'' and ``entries'' of a vector/matrix.) We define $S^-(\mathbf{0}):=0$. Next, allocate to each zero entries of $\mathbf{x}$ a value of 1 or $-1$, and denote by $S^+(\mathbf{x})$ the maximum possible number of sign changes in the resulting sequence. We set $S^+(\mathbf{0}):=n$ for $\mathbf{0}\in\mathbb{R}^n.$
			\item A \textit{contiguous submatrix} is one whose rows and columns are indexed by consecutive integers.
			\item Let $\adj(A)$ denote the adjugate matrix of $A$. If $n=1$, we define $\adj (A):=1$ to be the determinant of the empty matrix.
			\item Let $\mathbf{e}^j$ denote the unit vector in $\mathbb{R}^n$ whose $j^{th}$ entry is 1 and the rest are zero.
			\item The \textit{signature} or \textit{sign pattern} of an SSR (SR) matrix $A \in \mathbb{R}^{m \times n}$ is the ordered tuple $\epsilon = (\epsilon_1, \dots, \epsilon_{\min\{m,n\}})$ in a definition above, with $\epsilon_r$ denoting the sign of all $r \times r$ minors of $A$. We set $\epsilon_0 := 1$.
			\item Given integers $m,n\geq k\geq 1$ and a sign pattern $\epsilon=(\epsilon_1,\ldots,\epsilon_k)$, an SSR$_k$ (SR$_k$) matrix $A\in \mathbb{R}^{m\times n}$ is called SSR$_k({\epsilon})$ (SR$_k({\epsilon})$) if the sign pattern of all non-zero minors of $A$ of order at most $k$ is given by $\epsilon$. If $k=\min\{m,n\}$ then we simply write SSR$(\epsilon)$ (SR$(\epsilon)$).
			\item Let $A\begin{pmatrix}
				i_1,\ldots,i_p \\
				j_1,\ldots,j_q
			\end{pmatrix}$ denote the $p\times q$ submatrix of an $m\times n$ matrix $A$ indexed by rows $1\leq i_1<\cdots<i_p\leq m$ and columns $1\leq j_1<\cdots<j_q\leq n$, where $1\leq p\leq m$ and $1\leq q\leq n$.
		\end{enumerate}
	\end{defn}
	We now state our main results. The first characterizes SSR matrices of arbitrary sizes using the variation diminishing property, parallel to Motzkin's 1936 result for SR matrices. 
	\begin{utheorem}\label{A}
		Given $A\in\mathbb{R}^{m\times n}$,  $A$ satisfies the variation diminishing property, $S^+(A\mathbf{x})\leq S^-(\mathbf{x})$ for all $\mathbf{0}\neq\mathbf{x}\in\mathbb{R}^n$ if and only if $A$ is SSR.
	\end{utheorem} 
	The above theorem ensures that if a matrix $A$ satisfies the variation diminishing property then it is SSR. However, it does not give any information about the signature/sign pattern of $A$. Our next result aims to fill this gap, and it guarantees that we can obtain an SSR matrix with a sign pattern of desired choice at the cost of imposing a natural condition along with the variation diminishing property, which we have stated in part (2) of Theorem \ref{B}. Further, we have refined part (2) below in part (3) where exactly one vector $\mathbf{x}^{A_k}$ for every contiguous square submatrix $A_k$ of $A$ is sufficient to ensure that $A$ is SSR$(\epsilon)$. Thus, (3) reveals the extent to which (2) can be weakened without losing content.
	\begin{utheorem}\label{B}
		Given $A\in\mathbb{R}^{m\times n}$ and $\epsilon=(\epsilon_1,\ldots,\epsilon_{\mathrm{min}\{m,n\}})$, the following statements are equivalent.
		\begin{itemize}
			\item[(1)] $A$ is SSR$({\epsilon})$.
			\item[(2)] For all $\mathbf{0}\neq\mathbf{x}\in\mathbb{R}^n$, we have $S^+(A\mathbf{x})\leq S^-(\mathbf{x})$. Further, for all $\mathbf{x}\in\mathbb{R}^n$ with $A\mathbf{x}\neq\mathbf{0}$ and $S^+(A\mathbf{x})=S^-(\mathbf{x})=r$, if $0\leq r\leq \mathrm{min}\{m,n\}-1$, then the sign of the first (last) component of $A\mathbf{x}$ (if zero, the unique sign given in determining $S^+(A\mathbf{x})$) agrees with $\epsilon_r\epsilon_{r+1}$ times the sign of the first (last) non-zero component of $\mathbf{x}$. \smallskip
			
			In fact, these are also equivalent to the following assertion with a severely reduced test set:\smallskip
			
			\item[(3)] For every contiguous square submatrix $A_k$ of $A$ of size $k\times k$, where $1\leq k\leq \mathrm{min}\{m,n\}$, and given any fixed vector $\mathbf{0}\neq\mathbf{v}:=(\alpha_1,-\alpha_2,\ldots,(-1)^{k-1}\alpha_k)^T\in\mathbb{R}^k$ with all $\alpha_j\geq0$, we define the vector
			\begin{equation}\label{x^A_k}
				\mathbf{x}^{A_k}:=\adj(A_k)\mathbf{v}.
			\end{equation}
			Then $S^+(A_k\mathbf{x}^{A_k})\leq S^-(\mathbf{x}^{A_k})$. Moreover, 
			if $S^+(A_k\mathbf{x}^{A_k})=S^-(\mathbf{x}^{A_k})=r$, and if $0\leq r\leq k-1$, then the sign of the first (last) component of $A_k\mathbf{x}^{A_k}$ (if zero, the unique sign given in determining $S^+(A_k\mathbf{x}^{A_k})$) agrees with $\epsilon_{r}\epsilon_{r+1}$ times the sign of the first (last) non-zero component of $\mathbf{x}^{A_k}$.
		\end{itemize}
	\end{utheorem}
	\begin{rem}
		As a consequence of Theorem \ref{B}, one can reduce the test set in Theorem \ref{A} as well, to a single vector $	\mathbf{x}^{A_k}:=\adj(A_k)\mathbf{v}$ for each contiguous square submatrix, and where $\mathbf{v}$ is as in Theorem \ref{B}, with alternate signed coordinates, i.e.\- from the alternating bi-orthant. As we show in Theorem \ref{SSR_test_vector_not_biorthant}, single test vectors from other bi-orthant do not work.
	\end{rem}
    Our final result is a ``non-strict" variant of Theorem \ref{B} for SR matrices, which in particular strengthens Motzkin's result for each fixed sign pattern $\epsilon$ by removing Motzkin's rank-hypothesis. In fact we show a twofold strengthening of Motzkin's result, by moreover working with a single test vector for each square submatrix.

	\begin{utheorem}\label{VD_SRE}
	Given $A\in\mathbb{R}^{m\times n}$ and $\epsilon=(\epsilon_1,\ldots,\epsilon_{\mathrm{min}\{m,n\}})$, the following are equivalent. 
	\begin{itemize}
		\item[(1)] $A$ is SR$({\epsilon})$. 
		\item[(2)] For all $\mathbf{x}\in\mathbb{R}^n$, we have $S^-(A\mathbf{x})\leq S^-(\mathbf{x})$. Moreover, for all $\mathbf{x}\in\mathbb{R}^n$ with $A\mathbf{x}\neq\mathbf{0}$ and $S^-(A\mathbf{x})=S^-(\mathbf{x})=r$, if $0\leq r\leq \mathrm{min}\{m,n\}-1$, then the sign of the first (last) non-zero component of $A\mathbf{x}$ agrees with $\epsilon_r\epsilon_{r+1}$ times the sign of the first (last) non-zero component of $\mathbf{x}$. 
		\item[(3)] For every square submatrix $A_k$ of $A$ of size $k\times k$, where $1\leq k\leq\mathrm{min}\{m,n\}$, and given any fixed vector $\boldsymbol{\alpha}:=(\alpha_1,-\alpha_2,\ldots,(-1)^{k-1}\alpha_k)^T\in\mathbb{R}^k$ with all $\alpha_j>0$, we define the vector \[\mathbf{y}^{A_k}:=\adj(A_k)\boldsymbol{\alpha}.\] Then $S^-(A_k\mathbf{y}^{A_k})\leq S^-(\mathbf{y}^{A_k})$. Moreover, 
		if $S^-(A_k\mathbf{y}^{A_k})=S^-(\mathbf{y}^{A_k})=r$, and if $0\leq r\leq k-1$ and $A_k\mathbf{y}^{A_k}\neq\mathbf{0}$, then the sign of the first (last) non-zero component of $A_k\mathbf{y}^{A_k}$ agrees with $\epsilon_r\epsilon_{r+1}$ times the sign of the first (last) non-zero component of $\mathbf{y}^{A_k}$. 
	\end{itemize}
\end{utheorem}

	In a sense, Theorems \ref{A}, \ref{B}, and Theorem \ref{VD_SRE} are the culmination of many previous results, extending work on SSR/SR matrices (Schoenberg \cite{S30}, Motzkin \cite{Mot36} and Gantmacher--Krein \cite{GK50}) and TP/TN matrices (Gantmacher--Krein \cite{GK50}, Brown--Johnstone--MacGibbon \cite{BJM81} and Choudhury \cite{C22,C23}).\medskip

	\section{The Proofs} \label{section VDP} 	
	We begin by proving Theorems \ref{A} and \ref{B} -- for SSR matrices.
	\begin{proof}[Proof of Theorem \ref{A}] 
		Let $A\in\mathbb{R}^{m\times n}$ be such that $S^+(A\mathbf{x})\leq S^-(\mathbf{x})$ for all $\mathbf{0}\neq\mathbf{x}\in\mathbb{R}^n$. Our aim is to show all minors of the same size of $A$ have the same sign. The proof is by induction on the size of the minors. First we show that all $p\times p$ minors of $A$ are non-zero for $1\leq p\leq\mathrm{min}\{m,n\}$.  Assume to the contrary that $\det A\begin{pmatrix}
			i_1,\ldots,i_p \\
			j_1,\ldots,j_p
		\end{pmatrix}=0$, where (henceforth) $1\leq i_1<\cdots<i_p\leq m$ index the rows of the submatrix and $1\leq j_1<\cdots<j_p\leq n$ index the columns. Then there exists $\mathbf{0}\neq\mathbf{z}\in\mathbb{R}^{p}$ such that
		$A\begin{pmatrix}
			i_1,\ldots,i_p \\
			j_1,\ldots,j_p
		\end{pmatrix}\mathbf{z}=\mathbf{0}.$
		Define $\mathbf{x}\in\mathbb{R}^n$ such that the $j_l$ component of $\mathbf{x}$ is $z_{l}$ for $l=1,\ldots,p$, and all other components are zero. Then $S^-(\mathbf{x})\leq p-1$ while $S^+(A\mathbf{x})\geq p$, a contradiction. This shows the claim. 
		
		Next we show that all entries of $A$ are of the same sign. For each $j\leq n$, $S^-(\mathbf{e}^j)=0$, and thus $S^+(A\mathbf{e}^j)=0.$ Therefore, all the elements of $A$ present in the $j^{th}$ column (which we denote by $\mathbf{a}^j$ henceforth) are non-zero and moreover of the same sign. Now, we will show that no two columns of $A$ have different signs (so $n\geq 2$ here). On the contrary, assume without loss of generality, that $\mathbf{a}^{i_1}$ is positive while $\mathbf{a}^{i_2}$ is negative. We can choose positive real numbers $x_{i_1}$ and $x_{i_2}$ such that at least one entry of the vector $\mathbf{v} = x_{i_1}\mathbf{a}^{i_1}+x_{i_2}\mathbf{a}^{i_2}\in\mathbb{R}^{m}$ is zero. Take $\mathbf{0}\neq\mathbf{x}\in\mathbb{R}^n$ with $i_1$ and $i_2$ entries as $x_{i_1}$ and $x_{i_2}$, respectively and all other entries zero. Then $S^-(\mathbf{x})=0$ whereas  $S^+(A\mathbf{x}) = S^+(\mathbf{v}) \geq1$, a contradiction. Thus $A$ is SSR$_1.$ 
		
		Next, we assume that $A$ is SSR$_{p-1}$ for $1<p\leq\mathrm{min}\{m,n\}$, and show that $A$ is SSR$_p$. The proof strategy is broadly as follows: To prove all $p\times p$ minors of $A$ have the same sign, we will choose arbitrary $p$ columns of $A$, say indexed by $1\leq j_1<\cdots<j_p\leq n$ and then show that (a) all $p\times p$ minors of the submatrix of $A$ included in rows $1,\ldots,m$ and columns $j_1,\ldots,j_p$ have the same sign, and (b) this sign is independent of $j_1,\ldots,j_p$ and depends only on $p$.
		
		Fix $1\leq j_1<\cdots<j_p\leq n$. We will first prove that $\det A\begin{pmatrix}
			i_1,\ldots,i_p \\
			j_1,\ldots,j_p
		\end{pmatrix}$ for all $1\leq i_1<\cdots<i_p\leq m$ have the same sign. If $p=m$, then we are done; if $p<m$, it suffices to show this for the $(p+1)$ minors of size $p\times p$ that can be formed from an arbitrary choice of $(p+1)$ rows, $i_1<\cdots<i_{p+1}$ (say). Consider the following $p\times p$ submatrices of $A$ which are obtained by deleting the $l^{th}$ row (indexed by $i_l$) of $A\begin{pmatrix}
			i_1,\ldots,i_{p+1} \\
			j_1,\ldots,j_p
		\end{pmatrix}$ for $l=1,\ldots,p+1$:
	
		\[A_l=\begin{pmatrix}
			a_{i_1j_1} & a_{i_1j_2} & \ldots & a_{i_1j_p} \\
			\vdots    & \vdots    & \ddots & \vdots \\
			a_{i_{l-1}j_1} & a_{i_{l-1}j_2} & \ldots & a_{i_{l-1}j_p} \\
			a_{i_{l+1}j_1} & a_{i_{l+1}j_2} & \ldots & a_{i_{l+1}j_p} \\
			\vdots    & \vdots    & \ddots & \vdots \\
			a_{i_{p+1}j_1} & a_{i_{p+1}j_2} & \ldots & a_{i_{p+1}j_p} \\
		\end{pmatrix}_{p\times p}.\] We will show that $\det A_1=\det A_l$ for $l=2,\ldots,p+1$. Fix $l\neq1$; then the determinant of the $p\times p$ matrix $A_l$ is given by 
		\[\det A_l= a_{i_1j_1}A_l^{11}-a_{i_1j_2}A_l^{12}+\cdots+(-1)^{p-1}a_{i_ 1j_p}A_l^{1p},\]
		where $A_l^{11},\ldots,A_l^{1p}$ are non-zero (from above) $(p-1)\times(p-1)$ minors of $A$ and hence by the induction hypothesis they all have the same sign. Define $\mathbf{0}\neq\mathbf{x}\in\mathbb{R}^n$ such that the entries in positions $j_1,j_2,\ldots,j_p$ are $A_l^{11},-A_l^{12},\ldots,(-1)^{p-1}A_l^{1p}$, respectively and zero elsewhere. Thus $S^-(\mathbf{x})=p-1$. Now, the $i_1,\ldots,i_{p+1}$ entries of the vector $A\mathbf{x}$ are given by 
		\[\det A_l,0,\ldots,0,(-1)^l\det A_1,0,\ldots,0,\]
		where $(-1)^l\det A_1$ is present in the $i_l^{th}$ position. Suppose that $\det A_1$ and $\det A_l$ are of opposite signs. Replacing the $0$ entry by $(-1)^k(\det A_1)$ in positions $i_k$ for $k\in[2,p+1]\setminus\{l\}$, we see that $S^+(A\mathbf{x})\geq p$, which is not possible. Thus $\det A_1=\det A_l$ for all $l=2,\ldots,p+1.$ Therefore, the minors $\det A\begin{pmatrix}
			i_1,\ldots,i_p \\
			j_1,\ldots,j_p
		\end{pmatrix}$ for all $1\leq i_1<\cdots<i_p\leq m$ have the same signs. We denote this common sign of the minors by $\varepsilon^{\prime}(j_1,\ldots,j_p)$. Note that if $p=n$, then we are done. To complete the proof, we assume $p<n$ and show that this sign is independent of the particular choice of $p$ columns $j_1,\ldots,j_p$ and depends only on $p$. Let us take $(p+1)$ columns of $A$, with $1\leq j_1<\cdots<j_{p+1}\leq n$, and let $\varepsilon^{\prime}_r:=\varepsilon^{\prime}(j_1,\ldots,j_{r-1},j_{r+1},\ldots,j_{p+1}).$ It suffices to prove $\varepsilon^{\prime}_r=\varepsilon^{\prime}_{r+1}$ for $r=1,\ldots,p.$
		
		Define a vector $\mathbf{x}\in\mathbb{R}^{n}$ whose coordinates in position $j_r$ are given by
		\[x_{j_r}:=(-1)^{r-1}\det A\begin{pmatrix}
			1,2,\ldots \hspace{1.1cm} \ldots, p-1, p \\
			j_1,\ldots,j_{r-1},j_{r+1},\ldots,j_{p+1}
		\end{pmatrix}\]
		for $r=1,\ldots,p+1$ and zero elsewhere. By appending any row of $A\begin{pmatrix}
			1,2,\ldots \hspace{1.1cm} \ldots, p-1, p \\
			j_1,\ldots,j_{r-1},j_{r+1},\ldots,j_{p+1}
		\end{pmatrix}$ to the top, and taking the determinant of this singular matrix by expanding along the top row, it follows that
		\[\sum_{r=1}^{p+1}a_{ij_r}x_{j_r}=0 \quad\mathrm{for}\hspace{.1cm} i=1,\ldots,p\] 
		and hence $S^+(A\mathbf{x})\geq p$. Now if $\varepsilon^{\prime}_r\varepsilon^{\prime}_{r+1}<0$ for some $r$, then $x_{j_r}x_{j_{r+1}}>0$, and so $S^-(\mathbf{x})\leq p-1$ which is a contradiction. Hence, $A$ is SSR.
		
		To prove the converse, let $A\in\mathbb{R}^{m\times n}$ be SSR.  Take a non-zero vector $\mathbf{x}\in\mathbb{R}^n$ and assume that $0\leq S^-(\mathbf{x})=r\leq n-1$. Therefore, we can  partition $\mathbf{x}$ into $(r+1)$  contiguous components such that no two components having different signs belong to the same partition:
		\begin{align}\label{partition}
			(x_1,\ldots,x_{s_1}), \quad(x_{s_1+1},\ldots,x_{s_2}), \quad\ldots,\quad (x_{s_r+1},\ldots,x_n).
		\end{align}
		We assume without loss of generality that the sign of the non-zero elements in the $k^{th}$ partition is given by $(-1)^{k-1}$. Also, note that there is at least one non-zero element in each partition, otherwise $S^-(\mathbf{x})<r$. Writing $A=(\mathbf{a}^1|\ldots|\mathbf{a}^n)$ where $\mathbf{a}^i\in\mathbb{R}^m$ for all $i=1,\ldots, n$, we obtain 
		\[A\mathbf{x}=\sum_{k=1}^{n}x_k\mathbf{a}^k =\sum_{j=1}^{r+1}(-1)^{j-1}\mathbf{y}^j\] where \[\mathbf{y}^j=\sum_{k=s_{j-1}+1}^{s_j}|x_k|\mathbf{a}^k \hspace{.2cm}\mathrm{for}\hspace{.2cm} j=1,\ldots,r+1,\hspace{.2cm}\hspace{.2cm} s_0=0, \hspace{.2cm}\mathrm{and} \hspace{.2cm} s_{r+1}=n.\] 
		Let $Y:=(\mathbf{y}^1|\ldots|\mathbf{y}^{r+1})\in\mathbb{R}^{m\times (r+1)}$. Using basic properties of determinants, we have for $1\leq p\leq\mathrm{min}\{m, r+1\}:$
		\[\det Y\begin{pmatrix}
			i_1,\ldots,i_p \\
			j_1,\ldots,j_p
		\end{pmatrix}=\sum_{k_1=s_{j_1-1}+1}^{s_{j_1}}\cdots\sum_{k_p=s_{j_p-1}+1}^{s_{j_p}}|x_{k_1}|\cdots|x_{k_p}|\det A\begin{pmatrix}
			i_1,\ldots,i_p \\
			k_1,\ldots,k_p
		\end{pmatrix}. \] Since $A$ is SSR, the terms $\det A\begin{pmatrix}
			i_1,\ldots,i_p \\
			k_1,\ldots,k_p
		\end{pmatrix}$ have the same sign for all choices of $1\leq i_1<\cdots<i_p\leq m$, $1\leq k_1<\cdots<k_p\leq n$; and $x_{k_1}\cdots x_{k_p}\neq0$ for some choice of admissible $\{k_1,\ldots,k_p\}$ in the above sum. Hence $Y$ is SSR. Note that the minors of $A$ and $Y$ have the same sign pattern.
		
		With the above analysis in hand, we will show that $S^+(A\mathbf{x})\leq r$. Let $\mathbf{d}_{r+1}:=(1,-1,\ldots,(-1)^{r})^T\in\mathbb{R}^{r+1}$. Consider the following two cases.
		
		\noindent \textbf{Case I.} $m\leq r+1$.
		
		If $m=r+1$, then since $Y$ is SSR and hence invertible, hence $A\mathbf{x}=Y\mathbf{d}_{r+1}\neq\mathbf{0}$, and $S^+(A\mathbf{x}) = S^+(Y\mathbf{d}_{r+1})\leq r = S^-(\mathbf{x})$. Otherwise, $m\leq r$. Then clearly $S^+(A\mathbf{x})\leq m\leq S^-(\mathbf{x}).$
		
		\noindent \textbf{Case II.} $m>r+1$.
		
		Define $\mathbf{w}:=A\mathbf{x}=Y\mathbf{d}_{r+1}$ and note that $\mathbf{w}\neq\mathbf{0}$, since $Y$ has rank $(r+1)$ and hence independent columns. Now assume to the contrary that $S^+(A\mathbf{x})\geq r+1$. Thus there exist indices $1\leq i_1<\cdots<i_{r+2}\leq m$ and $\theta\in\{1,-1\}$ such that $\theta(-1)^{j-1}w_{i_j}\geq 0$ for $j=1,\ldots,r+2$. Since $Y$ is SSR$_{r+1}$, at most $r$ of the $w_{i_j}$ can be zero; in particular, at least two of them are non-zero. Now, the determinant 
		\[\det\begin{pmatrix}
			w_{i_1} & y_{i_11} & \ldots & y_{i_1r+1} \\
			w_{i_2} & y_{i_21} & \ldots & y_{i_2r+1} \\
			\vdots  & \vdots   & \ddots & \vdots     \\
			w_{i_{r+2}} & y_{i_{r+2}1} & \ldots & y_{i_{r+2}r+1} 
		\end{pmatrix}\]
		vanishes since the first column is an alternating sum of the rest. Expanding this determinant along the first column gives 
		\[0=\sum_{j=1}^{r+2}(-1)^{j-1}w_{i_j}\det Y\begin{pmatrix} 
			i_1,\ldots,i_{j-1},i_{j+1},\ldots,i_{r+2} \\
			1, 2,\ldots \hspace{1cm}\ldots,r, r+1	
		\end{pmatrix}.	\]
		But the right hand side is non-zero, since $Y$ is SSR, $(-1)^{j-1}w_{i_j}$ has the same sign $\theta$ for $j=1,\ldots,r+2$, and not all $w_{i_j}$ are zero. This contradiction implies that $S^+(A\mathbf{x})\leq r=S^-(\mathbf{x})$.
	\end{proof}
	To prove our next main result, we recall a 1968 seminal result by Karlin for strict sign regularity, first proved in 1912 by Fekete \cite{FP12} for total positivity and later refined in 1955 by Schoenberg \cite{S55} to total positivity of order $k$.
	\begin{theorem}[Karlin \cite{K68}]\label{Karlin}
		Suppose $1\leq k\leq m,n$ are integers and a matrix $A\in\mathbb{R}^{m\times n}$. Then $A$ is SSR$_k$ if and only if all contiguous minors of order $r\in\{1,\ldots,k\}$ have the same sign.
	\end{theorem}
	\begin{proof}[Proof of Theorem \ref{B}] 
		We will show a cyclic chain of implications: $(1) \implies (2) \implies (3) \implies (1)$. In addition, we also show that $(2) \implies (1)$ (although it is not needed) in keeping with previous proofs of the equivalence of SSR and VD. This will also be the strategy in our proof of Theorem \ref{VD_SRE} below, for SR matrices.
		
		We begin by showing (1)$\implies$(2). If $A\in\mathbb{R}^{m\times n}$ is SSR$(\epsilon)$, that $S^+(A\mathbf{x})\leq S^-(\mathbf{x})$ for all $\mathbf{0}\neq\mathbf{x}\in\mathbb{R}^n$ immediately follows from Theorem \ref{A}. It only remains to show the second part of the assertion. Let $S^+(A\mathbf{x})=S^-(\mathbf{x})=r$ with $A\mathbf{x}\neq\mathbf{0}$ and $0\leq r\leq\mathrm{min}\{m,n\}-1$. To proceed, we adopt the notation in the second subcase of the preceding proof's converse part; note that now there exist $r+1$ indices instead of $r+2$. We claim that if $\theta(-1)^{j-1}w_{i_j}\geq0$ for $j=1,\ldots,r+1$, then $\theta=\epsilon_r\epsilon_{r+1}$. To show this, we use the following system of equations
		\begin{equation}\label{thrmBeq1}
	Y\begin{pmatrix}
			i_1,\ldots,i_{r+1} \\
			1,\ldots,r+1
		\end{pmatrix}\mathbf{d}_{r+1}=\begin{pmatrix}
			w_{i_1} \\
			\vdots \\
			w_{i_{r+1}}
		\end{pmatrix},
		\end{equation}
		where $\mathbf{d}_{r+1}:=(1,-1,\ldots,(-1)^{r})^T$. Since $Y\in\mathbb{R}^{m\times(r+1)}$ is SSR, every set of $(r+1)$ rows of $Y$ is linearly independent. Using Cramer's rule to solve the system \eqref{thrmBeq1} for the first component gives
		\[1=\frac{\det\begin{pmatrix}
				w_{i_1} & y_{i_12} & \ldots & y_{i_1r+1} \\
				w_{i_2} & y_{i_22} & \ldots & y_{i_2r+1} \\
				\vdots  & \vdots   & \ddots & \vdots      \\
				w_{i_{r+1}} & y_{i_{r+1}2} & \ldots & y_{i_{r+1}r+1} 
		\end{pmatrix}}{\det Y\begin{pmatrix}
				i_1,\ldots,i_{r+1} \\
				1,\ldots,r+1
		\end{pmatrix}}.\]
		Expanding the numerator along the first column gives
		\[1=\frac{\displaystyle{\sum_{j=1}^{r+1}} (-1)^{j-1}w_{i_j}\det Y\begin{pmatrix}
				i_1,\ldots,i_{j-1},i_{j+1},\ldots,i_{r+1} \\
				2,3,\ldots \hspace{1cm} \ldots,r,r+1
		\end{pmatrix}}{\det Y\begin{pmatrix}
				i_1,\ldots,i_{r+1} \\
				1,\ldots,r+1
		\end{pmatrix}}.\]
		Note that $\theta(-1)^{j-1}w_{i_j}\geq0$ for $i=1,\ldots,r+1$ and all $p\times p$ minors of $Y$ have sign $\epsilon_p$ for $p=1,\ldots,r+1$. Multiplying both sides of the above equation by $\theta$, we have $\theta=\epsilon_r\epsilon_{r+1}$. But $\theta$ is the sign of $w_{i_1}$, i.e. of the first non-zero coordinate in an $S^+$-completion of $A\mathbf{x}$; and by \ref{partition} the first non-zero coordinate in $\mathbf{x}$ is positive. This shows~(2).
		
		We next show that (2)$\implies$(1), by induction on the size of the minors. Let $\mathbf{e}^1, \ldots,\mathbf{e}^n \in\mathbb{R}^n$ denote the standard basis. Then
		$S^+(A\mathbf{e}^j)\leq S^-(\mathbf{e}^j)=0$, so $A\mathbf{e}^j$ can not have zero components -- in fact all components have sign $\epsilon_1$ by (2).
		
		We next claim that all minors of $A$ are non-zero. Indeed, suppose for contradiction that $\det A\begin{pmatrix}
			i_1,\ldots,i_p \\
			j_1,\ldots,j_p
		\end{pmatrix}=0$, where $2\leq p\leq\mathrm{min}\{m,n\}$, $1\leq i_1<\cdots<i_p\leq m$ and $1\leq j_1<\cdots<j_p\leq n$. Then there exists $\mathbf{0}\neq\mathbf{z}\in\mathbb{R}^{p}$ such that
		$A\begin{pmatrix}
			i_1,\ldots,i_p \\
			j_1,\ldots,j_p
		\end{pmatrix}\mathbf{z}=\mathbf{0}.$
		Define $\mathbf{x}\in\mathbb{R}^n$ such that the $j^{th}_l$ component of $\mathbf{x}$ is $z_{l}$ for $l=1,\ldots,p$, and all other components are zero. Then $S^-(\mathbf{x})\leq p-1$ while $S^+(A\mathbf{x})\geq p$, a contradiction. Thus all minors of $A$ are non-zero. Finally, we claim by induction on $1\leq p\leq\mathrm{min}\{m,n\}$ that $A$ is SSR$_p$ with sign pattern $(\epsilon_1,\ldots,\epsilon_p)$, with the base case $p=1$ shown above. For the induction step, suppose $1\leq i_1<\cdots<i_p\leq m$ and $1\leq j_1<\cdots<j_p\leq n$ as above. Since $A\begin{pmatrix}
			i_1,\ldots,i_p \\
			j_1,\ldots,j_p
		\end{pmatrix}$ is non-singular, there exists $\mathbf{z}=(z_1,\ldots,z_p)^T\in\mathbb{R}^p$ such that
		\begin{equation}\label{Az=d}
			A\begin{pmatrix}
				i_1,\ldots,i_p \\
				j_1,\ldots,j_p
			\end{pmatrix}\mathbf{z}=\mathbf{d}_{p}.
		\end{equation}
		Again, extend $\mathbf{z}$ to $\mathbf{x}\in\mathbb{R}^n$ by embedding in positions $j_1,\ldots,j_p$ and padding by zeros elsewhere. Then $p-1\leq S^-(A\mathbf{x})\leq S^+(A\mathbf{x})\leq S^-(\mathbf{x}) = S^-(\mathbf{z}) \leq p-1$. It follows that $S^+(A\mathbf{x})=S^-(\mathbf{z})=p-1$. From this we conclude:
		(i) the coordinates of $\mathbf{z}$ alternate in sign; (ii) all coordinates of $A\mathbf{x}$ in positions $1,\ldots,i_1$ are positive; and (iii) $\epsilon_{p-1}\epsilon_pz_1>0$ by the hypothesis. We now return to equation \eqref{Az=d} and solve for $z_{1}$ using Cramer's rule to obtain 
		\[z_{1}=\frac{\displaystyle{\sum_{l=1}^{p}}\det A\begin{pmatrix}
				i_1,\ldots,i_{l-1},i_{l+1},\ldots,i_p \\
				j_2,j_3, \ldots \hspace{.2cm} \ldots, j_{p-1},j_p
		\end{pmatrix}}{\det A\begin{pmatrix}
				i_1,\ldots,i_p \\
				j_1,\ldots,j_p
		\end{pmatrix}}.\]
		Multiplying both sides by $\epsilon_{p-1}\epsilon_p$, \[0<\epsilon_{p-1}\epsilon_pz_{1}=\epsilon_{p-1}\epsilon_p\frac{\displaystyle{\sum_{l=1}^{p}}\det A\begin{pmatrix}
				i_1,\ldots,i_{l-1},i_{l+1},\ldots,i_p \\
				j_2,j_3, \ldots \hspace{.2cm} \ldots, j_{p-1},j_p
		\end{pmatrix}}{\det A\begin{pmatrix}
				i_1,\ldots,i_p \\
				j_1,\ldots,j_p
		\end{pmatrix}}.\]
		The sign of the numerator is $\epsilon_{p-1}$ by the induction hypothesis, and hence the sign of $\det A\begin{pmatrix}
			i_1,\ldots,i_p \\
			j_1,\ldots,j_p
		\end{pmatrix}$ is $\epsilon_p$. This concludes (2)$\implies$(1).
		
		Now we will show that (2)$\implies$(3). In fact we deduce from~(2) a strengthening of~(3), where $A_k$ is not necessarily required to be a contiguous submatrix of $A$ and $\mathbf{0}\neq\mathbf{x}^{A_k}\in\mathbb{R}^k$ can be arbitrary. Let $A_k=A\begin{pmatrix}
			i_1,\ldots,i_k \\
			j_1,\ldots,j_k
		\end{pmatrix}$, where $1\leq i_1<\cdots<i_k\leq m$ and $1\leq j_1<\cdots<j_k\leq n$. Define $\mathbf{x}\in\mathbb{R}^n$ whose coordinates are $x_l^{A_k}$ at position $j_l$ for $l=1,\ldots,k$ and zero elsewhere. By the hypothesis, we have 
		\[S^+(A_k\mathbf{x}^{A_k})\leq S^+(A\mathbf{x})\leq S^-(\mathbf{x})=S^-(\mathbf{x}^{A_k})\quad \implies \quad S^+(A_k\mathbf{x}^{A_k})\leq S^-(\mathbf{x}^{A_k}).\]
		Now suppose that $S^+(A_k\mathbf{x}^{A_k})=S^-(\mathbf{x}^{A_k})=r$, where $0\leq r\leq k-1$. Then $A_k\mathbf{x}^{A_k}\neq\mathbf{0}$ and \[ S^+(A\mathbf{x})=S^-(\mathbf{x})=r.\] 
		Since $A_k\mathbf{x}^{A_k}\neq\mathbf{0}$ is a sub-string of $A\mathbf{x}$, hence $A\mathbf{x}\neq\mathbf{0}$. Also, as $S^+(A\mathbf{x})=S^+(A_k\mathbf{x}^{A_k})$ therefore by considering any $S^+$-filling of $A\mathbf{x}$, all coordinates of $A\mathbf{x}$ in positions $1,\ldots,i_1$ (respectively $i_{k},\ldots,m$) have the same sign. Note that the first non-zero component of $\mathbf{x}^{A_k}$ is same as that of $\mathbf{x}$. Since $S^+(A\mathbf{x})=S^-(\mathbf{x})$ and $A\mathbf{x}\neq\mathbf{0}$, hence by (2), the sign of the first (last) component of $A_k\mathbf{x}^{A_k}$ agrees with $\epsilon_{r}\epsilon_{r+1}$ times the sign of the first (last) non-zero component of $\mathbf{x}^{A_k}$. 
		
		To complete the proof, we show that (3)$\implies$(1). By Karlin's Theorem \ref{Karlin}, it suffices to show that the sign of $\det A_r$ is $\epsilon_r$ for all $r\times r$ contiguous submatrices $A_r$ of $A$, where $1\leq r\leq \mathrm{min}\{m,n\}$. We prove this by induction on $r$. If $r=1$, then $\adj (A_1)=(1)_{1\times1}$ and
		\[0\leq S^+(A_1x^{A_1})\leq S^-(x^{A_1})=0.\]
		Thus $A_1x^{A_1}\neq0$ and by the hypothesis the sign of $A_1x^{A_1}$ is $\epsilon_1$ times the sign of $x^{A_1}>0$. Hence all the entries of $A$ have sign $\epsilon_1$.
		
		For the induction step fix $1\leq r\leq\mathrm{min}\{m,n\}$, and suppose that all contiguous $p\times p$ minors of $A$ have sign $\epsilon_{p}$ for $1\leq p\leq r-1$. Let $A_r$ be an $r\times r$ contiguous submatrix of $A$. By Theorem \ref{Karlin}, $A_r$ is SSR$_{r-1}$. Therefore, all entries of $\adj (A_r)$ are non-zero and have a checkerboard sign pattern. Now, define a vector $\mathbf{x}^{A_r}:=\adj (A_r)\mathbf{v}$, as in \eqref{x^A_k}. Note that all entries of $\mathbf{x}^{A_r}$ are non-zero with alternating signs. We first note that $A_r$ is non-singular, because if not then \[A_r\mathbf{x}^{A_r}=(\det A_r)\mathbf{v}=\mathbf{0}\in\mathbb{R}^r\implies S^+(A_r\mathbf{x}^{A_r})=r>r-1=S^-(\mathbf{x}^{A_r}),\] 
		a contradiction.
		
		Next, we show that $\det A_r$ has sign $\epsilon_r$. Since $A_r\mathbf{x}^{A_r}=(\det A_r)\mathbf{v}$, we have \[r-1=S^-(\mathbf{x}^{A_r})\geq S^+(A_r\mathbf{x}^{A_r})= S^+((\det A_r)\mathbf{v}).\] 
		Now, note that even if some components of the vector $\mathbf{v}$ are zero, the conditions on the $\alpha_j$ imply that $\mathbf{v}$ can be $S^+$-completed to a vector with all non-zero components and alternating signs. In particular, $S^+((\det A_r)\mathbf{v})=r-1$. Thus, we have
		\[S^+(A_r\mathbf{x}^{A_r})=S^-(\mathbf{x}^{A_r})=r-1.\]
		Since the sign of the first component of $A_r\mathbf{x}^{A_r}$ (or $S^+$-completion of $A_r\mathbf{x}^{A_r}$) is sign($\det A_r$), therefore by (3), sign($\det A_r$) and the first (non-zero) component of $\mathbf{x}^{A_r}$ bear the following relation:
		\[\mathrm{sign}(\det A_r)= \epsilon_{r-1}\epsilon_r\mathrm{sign}(\sum_{j=1}^{r}\alpha_jA_r^{j1}),\]
		where $A_r^{ji}$ denotes the determinant of the $(r-1)\times(r-1)$ submatrix of $A_r$ formed by deleting the $j^{th}$ row and $i^{th}$ column of $A_r$. Observe that the sign of each summand on the right of the above equation is $\epsilon_{r-1}$, since sign$(A_r^{j1})=\epsilon_{r-1}$ for $j=1,\ldots,r$, $\alpha_j\geq0$ for $j=1,\ldots,r$, and not all $\alpha_j$ are zero. Hence sign$(\det A_r)=\epsilon_r$. This completes our proof.
	\end{proof}
    	From part (3) of Theorem \ref{B}, we have seen that given an $m\times n$ matrix $A$ and a sign pattern $\epsilon$, the variation diminishing (VD) property for each contiguous square submatrix of $A$, at a single test vector which is drawn from the alternating bi-orthant suffices to show the strict sign regularity of $A$ with sign pattern $\epsilon$. Our next result shows that Theorem \ref{B} is the ``best possible" in the following sense: Any $n\times n$ singular SSR$_{n-1}$ matrix also satisfies the VD property on every vector in $\mathbb{R}^n$ in the interior of a bi-orthant other than the alternating bi-orthant. Thus a characterization of strict sign regularity with a given sign pattern of $n\times n$ matrices in terms of the VD property can not hold with test vectors in any open non-alternating bi-orthant. The proof is similar to $(1)\implies(2)$ of Theorem \ref{B} and we briefly sketch the arguments in the proof.
    	
    	\begin{theorem} \label{SSR_test_vector_not_biorthant} Fix an integer $n>0$ and a sign pattern $\epsilon=(\epsilon_1,\ldots,\epsilon_{n-1})$. Let the square matrix $A\in\mathbb{R}^{n\times n}$ be SSR$_{n-1}(\epsilon)$. Then $S^+(A\mathbf{x}_0)\leq S^-(\mathbf{x}_0)$, for every vector $\mathbf{x}_0$ with no zero entries and at least two successive coordinates which have the same sign. Further, if $A\mathbf{x}_0\neq\mathbf{0}$ and $S^+(A\mathbf{x}_0)=S^-(\mathbf{x}_0)=r$ for some $0\leq r\leq n-2$, then the sign of the first (last) component of $A\mathbf{x}_0$ (if zero, the unique sign given in determining $S^+(A\mathbf{x}_0)$) equals $\epsilon_r\epsilon_{r+1}$ times the sign of the first (last) component of $\mathbf{x}_0$. 
    	\end{theorem}
    		\begin{proof}
    			
    			Write	$\mathbf{x}_0=(x_1,\ldots,x_n)^T$. Then $S^-(\mathbf{x}_0)=r\leq n-2$ and one can partition $\mathbf{x}_0$ as in \eqref{partition}. Since $A\in\mathbb{R}^{n\times n}$ is SSR$_{n-1}(\epsilon)$, by the argument following \eqref{partition}, the matrix $Y \in \mathbb{R}^{n \times {(r+1)}}$ is SSR with sign pattern $(\epsilon_1,\ldots,\epsilon_{r+1})$ and $A\mathbf{x}_0=Y\mathbf{d}_{r+1}$. We define $\mathbf{w}:=A\mathbf{x}_0=Y\mathbf{d}_{r+1}$. Since $n>r+1$, by repeating the argument of case II in the proof of Theorem \ref{A}, we have
    			\[S^+(A\mathbf{x}_0)\leq S^-(\mathbf{x}_0).\] 
    			
    			To show the second part of the assertion, we may repeat the argument used in the proof of $(1)\implies(2)$ of Theorem \ref{B}.
    				\end{proof}

    	It remains to show that for each positive integer $n$, we can always construct a real $n\times n$ matrix $A$ such that $A$ is SSR$_{n-1}$ while $\det A=0.$ This follows from a result of Gantmacher--Krein (see Theorem \ref{G-K} below). The steps to construct such a matrix are as follows:
	\begin{enumerate}[(i)]
		\item Take any $B\in\mathbb{R}^{(n-1)\times (n-1)}$ such that $B$ is SSR$(\epsilon)$ for a fixed sign pattern $\epsilon$. (Such matrices always exist.) Clearly,  $\rk(B)=n-1$.
		\item Let $B^{\prime}:=B\oplus\{0\}\in\mathbb{R}^{n\times n}$. Therefore, $B^{\prime}$ is an $n\times n$ SR matrix whose rank is $(n-1)$.
		\item Define $A:=F_{\sigma}^{(n)}B^{\prime}F_{\sigma}^{(n)}$ where $F_{\sigma}^{(n)}=(\exp^{-\sigma(i-j)^2})_{i,j=1}^n$ for $\sigma>0$ is TP. By the Cauchy--Binet formula, $A$ is SSR$_{n-1}(\epsilon)$ with rank $(n-1)$ and $\det A=0$.
	\end{enumerate}
	\subsection{Variation Diminution for Sign Regular Matrices}\label{sec_VDSR}

	We conclude by showing Theorem \ref{VD_SRE}, which characterizes all $m\times n$ sign regular matrices with a given sign pattern using the variation diminishing property. The proof requires a density theorem proved by Gantmacher--Krein in 1950 using the total positivity of the Gaussian kernel.
	\begin{theorem}[Gantmacher--Krein \cite{GK50}] \label{G-K}
		Let $m,n\geq k\geq1$ be integers. Given a sign pattern $\epsilon=(\epsilon_1,\ldots,\epsilon_k)$, the set of $m\times n$ SSR$_k(\epsilon)$ matrices is dense in the set of $m\times n$ SR$_k(\epsilon)$ matrices.
	\end{theorem}
	The following basic lemma on sign changes of limits of vectors will also be needed in the proof.
	\begin{lemma}\cite[Lemma 3.2]{pinkus}\label{liminf} 
		If $\displaystyle{\lim_{k\to\infty}}\mathbf{x}_k=\mathbf{x}\neq\mathbf{0}\in\mathbb{R}^n$, then \[\varliminf_{k\to\infty}S^-(\mathbf{x}_k)\geq S^-(\mathbf{x})
		\hspace{.2cm}\mathrm{and}\hspace{.2cm} \varlimsup_{k\to\infty}S^+(\mathbf{x}_k)\leq S^+(\mathbf{x}).\]
	\end{lemma}

With these tools at hand, we now show: 
\begin{proof}[Proof of Theorem \ref{VD_SRE}]
		We first show that (1)$\implies$(2). Since $A\in\mathbb{R}^{m\times n}$ is SR$(\epsilon)$, by Theorem \ref{G-K} there exists a sequence of $m\times n$ SSR$(\epsilon)$ matrices $A_l$ converging entrywise to $A$.
		For $\mathbf{0}\neq\mathbf{x}\in\mathbb{R}^n$, using Theorem \ref{A} and Lemma \ref{liminf}, we have \[S^-(A\mathbf{x})\leq\varliminf_{l\to\infty}S^-(A_l\mathbf{x})\leq \varliminf_{l\to\infty}S^+(A_l\mathbf{x})\leq \varliminf_{l\to\infty}S^-(\mathbf{x}) = S^-(\mathbf{x}).\]
		If $S^-(A\mathbf{x})=S^-(\mathbf{x})=r$ where $0\leq r\leq\mathrm{min}\{m,n\}-1$ and $A\mathbf{x}\neq\mathbf{0}$, then we necessarily have 
		\[r=S^-(A\mathbf{x})\leq S^-(A_l\mathbf{x})\leq S^+(A_l\mathbf{x})\leq S^-(\mathbf{x})=r\]
		for all $l$ sufficiently large, using Theorem \ref{A} and Lemma \ref{liminf}. Therefore for large $l$, $S^-(A_l\mathbf{x})=S^+(A_l\mathbf{x})$, i.e. the sign patterns of $A_l\mathbf{x}$ do not depend on any zero entries. At the same time, both vectors $\mathbf{x}$ and $A_l\mathbf{x}$ admit `partitions' of the type \eqref{partition} with precisely $r$ sign changes. Since $S^+(A_l\mathbf{x})=S^-(\mathbf{x})$ and $A_l\mathbf{x}\neq\mathbf{0}$, by Theorem \ref{B} we have that the non-zero sign pattern (sign pattern after removing all zero entries in $A_l\mathbf{x}$) of $A_l\mathbf{x}$ agrees with $\epsilon_r\epsilon_{r+1}$ times that of $\mathbf{x}$. By a limiting argument, the non-zero sign patterns of $A\mathbf{x}$ and $A_l\mathbf{x}$ agree.
		
		That (2)$\implies$(1) is shown similarly to the proof of Theorem \ref{B} by induction on the size $p\times p$, where $1\leq p\leq\mathrm{min}\{m,n\}$. Again observe that $S^-(A\mathbf{e}^j)\leq S^-(\mathbf{e}^j)=0$ for $1\leq j\leq n$. Since the first non-zero component of $\mathbf{e}^j$ is positive, by the hypothesis all non-zero components of $A\mathbf{e}^j$ have  sign $\epsilon_1$ and hence all non-zero entries of $A$ are of sign $\epsilon_1$.
		
		We now assume that all $(p-1)\times(p-1)$ minors of $A$ have  sign $\epsilon_{p-1}$, where $2\leq p\leq\mathrm{min}\{m,n\}$. Consider the $p\times p$ submatrix of $A$ indexed by rows $1\leq i_1<\cdots<i_p\leq m$ and columns $1\leq j_1<\cdots<j_p\leq n$. If the determinant of this submatrix is zero, then we are done. Therefore, assume that this minor is non-zero and hence there exists $\mathbf{z}=(z_{1},\ldots,z_{p})^T\in\mathbb{R}^p$ satisfying $A\begin{pmatrix}
			i_1,\ldots,i_p\\
			j_1,\ldots,j_p
		\end{pmatrix}\mathbf{z}=\mathbf{d}_p:=(1,-1,\ldots,(-1)^{p-1})^T.$ By repeating the corresponding part of the proof of Theorem \ref{B}, we have $S^-(A\mathbf{x})=S^-(\mathbf{x})=p-1$ and hence $A\mathbf{x}\neq\mathbf{0}$, which further implies $\epsilon_{p-1}\epsilon_pz_1>0$. Now by Cramer's rule,
		\[z_{1}=\frac{\displaystyle{\sum_{l=1}^{p}}\det A\begin{pmatrix}
				i_1,\ldots,i_{l-1},i_{l+1},\ldots,i_p \\
				j_2,j_3, \ldots \hspace{.2cm} \ldots, j_{p-1},j_p
		\end{pmatrix}}{\det A\begin{pmatrix}
				i_1,\ldots,i_p \\
				j_1,\ldots,j_p
		\end{pmatrix}}.\]
		Multiplying both sides by $\epsilon_{p-1}\epsilon_p$, we obtain \[0<\epsilon_{p-1}\epsilon_pz_{1}=\epsilon_{p-1}\epsilon_p\frac{\displaystyle{\sum_{l=1}^{p}}\det A\begin{pmatrix}
				i_1,\ldots,i_{l-1},i_{l+1},\ldots,i_p \\
				j_2,j_3, \ldots \hspace{.2cm} \ldots, j_{p-1},j_p
		\end{pmatrix}}{\det A\begin{pmatrix}
				i_1,\ldots,i_p \\
				j_1,\ldots,j_p
		\end{pmatrix}}.\]
		Thus the sum on the right must be non-zero. Moreover, by the induction hypothesis, it has sign $\epsilon_{p-1}$. Thus the sign of $\det A\begin{pmatrix}
			i_1,\ldots,i_p \\
			j_1,\ldots,j_p
		\end{pmatrix}$ is $\epsilon_p$. This completes the induction step. 
		
		Now we show that (2)$\implies$(3) -- again more generally, for arbitrary vectors $\mathbf{0}\neq\mathbf{y}^{A_k}\in\mathbb{R}^k$. The proof is similar to that of Theorem \ref{B}. Fix $1\leq k\leq\mathrm{min}\{m,n\}$, and let $A_k = A\begin{pmatrix}
		i_1, \cdots, i_k \\
		j_1, \cdots, j_k
		\end{pmatrix}$ be an  arbitrary $k\times k$ submatrix of $A$. Take $\mathbf{x}\in\mathbb{R}^n$ whose coordinates are $y_l^{A_k}$ at position $j_l$ for $l=1,\ldots,k$  and zero elsewhere. Then
		\[S^-(A_k\mathbf{y}^{A_k})\leq S^-(A\mathbf{x}) \leq S^-(\mathbf{x})= S^-(\mathbf{y}^{A_k}).\] 
		Now suppose $S^-(A_k\mathbf{y}^{A_k})=S^-(\mathbf{y}^{A_k})=r$, where $0\leq r\leq k-1$ with $A_k\mathbf{y}^{A_k}\neq\mathbf{0}$. Then \[S^-(A_k\mathbf{y}^{A_k})=S^-(\mathbf{y}^{A_k})=S^-(\mathbf{x})=S^-(A\mathbf{x})=r\hspace{.2cm}\mathrm{and}\hspace{.2cm}A\mathbf{x}\neq\mathbf{0}.\] 
		Assume without loss of generality that the first and the last non-zero components of $A_k\mathbf{y}^{A_k}$ are in positions $s,t\in\{1,\ldots,k\}$, respectively. Since $S^-(A\mathbf{x})=S^-(A_k\mathbf{y}^{A_k})=r$ and $A_k\mathbf{y}^{A_k}$ is a sub-string of $A\mathbf{x}$, all non-zero coordinates of $A\mathbf{x}$ in positions $1, 2,\ldots,i_s$ (respectively $i_t,\ldots,m$) have the same sign, which moreover agrees with $\epsilon_r\epsilon_{r+1}$ times the sign of the first (respectively last) non-zero component of $\mathbf{x}$ -- which is the sign of the first (respectively last) non-zero component of $\mathbf{y}^{A_k}$. 
		
		Finally we show that (3)$\implies$(1). We show by induction on $r$ that the sign of $\det A_r$ is $\epsilon_r$ for all $r\times r$ non-singular submatrices of $A$, where $1\leq r\leq\mathrm{min}\{m,n\}$. For the base case $r=1$, if $A_1=(0)_{1\times1}$, there is nothing to prove. Otherwise we use that $\adj(A_1)=(1)_{1\times1}$, which further implies $y^{A_1}=(\alpha_1)_{1\times1}$, where $\alpha_1>0$. Since $S^-(A_1y^{A_1})=S^-(y^{A_1})=0$ and $A_1y^{A_1}\neq0$, by the hypothesis the sign of $A_1y^{A_1}$ is $\epsilon_1$ times the sign of $y^{A_1}$. Therefore, all non-zero entries of $A$ have sign $\epsilon_1$.
		
		For the induction step, $A_r$ is an $r\times r$ submatrix of $A$, and $A$ is SR$_{r-1}$ with sign pattern $(\epsilon_1,\ldots,\epsilon_{r-1})$. If $\det A_r=0$, we are done; else henceforth assume $\det A_r\neq0$. Let \[\boldsymbol{\alpha}=(\alpha_1,-\alpha_2,\ldots,(-1)^{r-1}\alpha_r)^T\in\mathbb{R}^r\hspace{.2cm}\mathrm{with}\hspace{.2cm}\mathrm{all} \hspace{.2cm}\alpha_j>0\] and define $\mathbf{y}^{A_r}:=\adj(A_r)\boldsymbol{\alpha}.$ Note that no row of $\adj(A_r)$ is zero as $\det A_r\neq0$. Also, since $A_{r}$ is SR$_{r-1}(\epsilon)$  with $\epsilon=(\epsilon_1,\ldots,\epsilon_{r-1})$, $\adj (A_r)$ has entries whose signs are in a checkerboard pattern. Further, as all coordinates of $\boldsymbol{\alpha}$ are non-zero with alternating signs, it follows that $S^-(\mathbf{y}^{A_r})=r-1$. Since $\mathbf{0}\neq A_r\mathbf{y}^{A_r}=(\det A_r)\boldsymbol{\alpha},$ \[S^-(A_r\mathbf{y}^{A_r})=S^-(\mathbf{y}^{A_r})=r-1.\] 
		By the hypothesis, the sign of the first non-zero entry of $A_r\mathbf{y}^{A_r}$ and the first non-zero entry of $\mathbf{y}^{A_r}$ satisfy:
		\[\mathrm{sign}(\alpha_1\det A_r)=\epsilon_{r-1}\epsilon_r\mathrm{sign}(\sum_{j=1}^{r}\alpha_jA_r^{j1}),\] 
		where $A_r^{j1}$ is the $(j,1)$ minor of $A_r$ (as in the proof of Theorem \ref{B}). Thus sign$(\det A_r)=\epsilon_r$. This completes our proof.
	\end{proof}


	\section*{Acknowledgments}
	We thank Apoorva Khare for a detailed reading of an earlier draft and for providing valuable feedback. We are also grateful to the anonymous referee for carefully going through the manuscript and offering several constructive comments that helped improve the exposition. The first author was partially supported by INSPIRE Faculty Fellowship research grant DST/INSPIRE/04/2021/002620
	(DST, Govt.~of India), and IIT Gandhinagar Internal Project: IP/IITGN/MATH/PNC/2223/25.

\end{document}